\documentclass[10pt,a4paper,twoside]{article}
\usepackage[textsize=scriptsize]{todonotes}
\usepackage{longtable}

\usepackage{etoolbox}
\newtoggle{draft}
\newtoggle{final}

\usepackage{tikz}
\usetikzlibrary{matrix,arrows}
\usepackage{tikz-cd}
\usepackage{stmaryrd}
\usepackage{listings}

\toggletrue{final}

\newcommand{\Gr}{\mathrm{Gr}}
\newcommand{\hflf}[1]{{h^{\mathrm{flf}}(#1)}}
\newcommand{\hflfp}[1]{{h_+^{\mathrm{flf}}(#1)}}

\usepackage{etoolbox}
\iftoggle{draft} {
\usepackage[margin=1.5in]{geometry}
\geometry{a4paper}
}{ 
\setlength{\marginparwidth}{0.75in}
}

\usepackage{amsmath}
\usepackage{amssymb}
\usepackage{amsthm}
\usepackage{amscd}
\usepackage{enumerate}
\usepackage[pdfusetitle,unicode,hidelinks]{hyperref}
\usepackage{bbm}
\usepackage{etoolbox}

\usepackage[utf8]{inputenc}

\newtheorem{proposition}{Proposition}
\newtheorem{corollary}[proposition]{Corollary}
\newtheorem{lemma}[proposition]{Lemma}
\newtheorem{theorem}[proposition]{Theorem}

\newtheorem*{conjecture*}{Conjecture}
\newtheorem*{theorem*}{Theorem}
\newtheorem*{corollary*}{Corollary}
\newtheorem*{proposition*}{Proposition}
\newtheorem*{lemma*}{Lemma}
\theoremstyle{definition}
\newtheorem{definition}[proposition]{Definition}

\newtheorem*{definition*}{Definition}
\newtheorem*{construction*}{Construction}
\theoremstyle{remark}
\newtheorem{remark}[proposition]{Remark}
\newtheorem*{remark*}{Remark}

\newtheorem{example}[proposition]{Example}
\newtheorem*{example*}{Example}

\newcommand{\id}{\operatorname{id}}
\newcommand{\Z}{\mathbb{Z}}

\let\scr=\mathcal
\let\bb=\mathbb
\newcommand{\Gm}{{\mathbb{G}_m}}
\newcommand{\Gmp}[1]{{\mathbb{G}_m^{\wedge #1}}}
\def\A{\bb A}
\def\P{\bb P}

\newcommand{\1}{\mathbbm{1}}

\newcommand{\eff}{{\text{eff}}}
\newcommand{\veff}{{\text{veff}}}

\newcommand{\SH}{\mathcal{SH}}

\DeclareMathOperator*{\colim}{colim}

\let\lim=\relax
\DeclareMathOperator*{\lim}{lim}
\def\Map{\mathrm{Map}}
\def\iMap{\ul{\mathrm{Map}}}

\def\CMon{\mathrm{CMon}}

\def\PSh{\mathcal{P}}

\def\Spc{\mathcal{S}\mathrm{pc}{}}

\newcommand{\Spec}{\mathrm{Spec}}
\newcommand{\gp}{\mathrm{gp}}

\newcommand{\wequi}{\simeq}
\newcommand{\Mod}{\text{-}\mathcal{M}\mathrm{od}}
\def\adj{\leftrightarrows}

\DeclareRobustCommand{\ul}{\underline}

\def\op{\mathrm{op}}

\let\cat=\mathrm
\def\Sm{{\cat{S}\mathrm{m}}}

\def\Nis{\mathrm{Nis}}
\def\Zar{\mathrm{Zar}}

\def\mot{\mathrm{mot}}

\newcommand{\fr}{\mathrm{fr}}

\def\ph{\mathord-}

\numberwithin{proposition}{section}
\numberwithin{equation}{section}

\iftoggle{final} {
\renewcommand{\todo}[1]{}
\newcommand{\NB}[1]{}
}{ 

\newcommand{\NB}[1]{\todo[color=gray!40]{#1}}
}

\newcommand{\ratC}{\hat{\mathrm C}_1}
\def\EssSm{\Sm^{\text{ess}}}
\def\flf{\mathrm{flf}}
\def\fsyn{\mathrm{fsyn}}
\newcommand{\Cor}{\mathrm{Cor}}
\newcommand{\SHSC}{\SH^{S^1\scr C}}
\newcommand{\SHSfr}{\SH^{S^1\fr}}
\newcommand{\SHS}{\SH^{S^1}\!}
\newcommand{\kgl}{\mathrm{kgl}}
\newcommand{\KGL}{\mathrm{KGL}}

\def\YEAR{\year}\newcount\VOL\VOL=\YEAR\advance\VOL by-1995
\def\firstpage{1}\def\lastpage{1000}
\def\received{}\def\revised{}
\def\communicated{}

\makeatletter
\def\magnification{\afterassignment\m@g\count@}
\def\m@g{\mag=\count@\hsize6.5truein\vsize8.9truein\dimen\footins8truein}
\makeatother

\font\eightrm=cmr8
\font\caps=cmcsc10                    
\font\Caps=cmcsc10 scaled \magstep1   
\font\scaps=cmcsc8

\makeatletter
\@specialpagefalse\headheight=8.5pt
\def\DocMath{{\def\th{\thinspace}\scaps Documenta Math.}}
\renewcommand{\@oddfoot}{\hfill\scaps Documenta Mathematica 
    \number\VOL\  (\number\YEAR) \number\firstpage--\lastpage\hfill}
\renewcommand{\@evenfoot}{\ifnum\thepage>\lastpage\hfill\scaps
    Documenta Mathematica \number\VOL\  (\number\YEAR)\hfill\else\@oddfoot\fi}%

\renewcommand{\@evenhead}{%
    \ifnum\thepage>\lastpage\rlap{\thepage}\hfill%
    \else\rlap{\thepage}\slshape\leftmark\hfill{\caps\SAuthor}\hfill\fi}%
\renewcommand{\@oddhead}{%
    \ifnum\thepage=\firstpage{\DocMath\hfill\llap{\thepage}}%
    \else{\slshape\rightmark}\hfill{\caps\STitle}\hfill\llap{\thepage}\fi}%
\makeatother

\def\TSkip{\bigskip}
\newbox\TheTitle{\obeylines\gdef\GetTitle #1
\ShortTitle  #2
\SubTitle    #3
\Author      #4
\ShortAuthor #5
\EndTitle
{\setbox\TheTitle=\vbox{\baselineskip=20pt\let\par=\cr\obeylines%
\halign{\centerline{\Caps##}\cr\noalign{\medskip}\cr#1\cr}}%
	\copy\TheTitle\TSkip\TSkip%
\def\next{#2}\ifx\next\empty\gdef\STitle{#1}\else\gdef\STitle{#2}\fi%
\def\next{#3}\ifx\next\empty%
    \else\setbox\TheTitle=\vbox{\baselineskip=20pt\let\par=\cr\obeylines%
    \halign{\centerline{\caps##} #3\cr}}\copy\TheTitle\TSkip\TSkip\fi%
\centerline{\caps #4}\TSkip\TSkip%
\def\next{#5}\ifx\next\empty\gdef\SAuthor{#4}\else\gdef\SAuthor{#5}\fi%
\ifx\received\empty\relax
    \else\centerline{\eightrm Received: \received}\fi%
\ifx\revised\empty\TSkip%
    \else\centerline{\eightrm Revised: \revised}\TSkip\fi%
\ifx\communicated\empty\relax
    \else\centerline{\eightrm Communicated by \communicated}\fi\TSkip\TSkip%
\catcode'015=5}}\def\Title{\obeylines\GetTitle}
\def\Abstract{\begingroup\narrower
    \parskip=\medskipamount\parindent=0pt{\caps Abstract. }}
\def\EndAbstract{\par\endgroup\TSkip}

\long\def\MSC#1\EndMSC{\def\arg{#1}\ifx\arg\empty\relax\else
     {\par\narrower\noindent%
     2010 Mathematics Subject Classification: #1\par}\fi}

\long\def\KEY#1\EndKEY{\def\arg{#1}\ifx\arg\empty\relax\else
	{\par\narrower\noindent Keywords and Phrases: #1\par}\fi\TSkip}

\newbox\TheAdd\def\Addresses{\vfill\copy\TheAdd\vfill
    \ifodd\number\lastpage\vfill\eject\phantom{.}\vfill\eject\fi}
{\obeylines\gdef\GetAddress #1
\Address #2 
\Address #3
\Address #4
\EndAddress
{\def\xs{4.3truecm}\parindent=0pt
\setbox0=\vtop{{\obeylines\hsize=\xs#1\par}}\def\next{#2}
\ifx\next\empty 
     \setbox\TheAdd=\hbox to\hsize{\hfill\copy0\hfill}
\else\setbox1=\vtop{{\obeylines\hsize=\xs#2\par}}\def\next{#3}
\ifx\next\empty 
     \setbox\TheAdd=\hbox to\hsize{\hfill\copy0\hfill\copy1\hfill}
\else\setbox2=\vtop{{\obeylines\hsize=\xs#3\par}}\def\next{#4}
\ifx\next\empty\ 
     \setbox\TheAdd=\vtop{\hbox to\hsize{\hfill\copy0\hfill\copy1\hfill}
                \vskip20pt\hbox to\hsize{\hfill\copy2\hfill}}
\else\setbox3=\vtop{{\obeylines\hsize=\xs#4\par}}
     \setbox\TheAdd=\vtop{\hbox to\hsize{\hfill\copy0\hfill\copy1\hfill}
	        \vskip20pt\hbox to\hsize{\hfill\copy2\hfill\copy3\hfill}}
\fi\fi\fi\catcode'015=5}}\gdef\Address{\obeylines\GetAddress}

\def\LOCAL{\jobname.files}

\begin{document}

\Title
Cancellation theorem for motivic spaces with finite flat transfers
\ShortTitle 
Cancellation theorem for finite flat correspondences
\SubTitle   
\Author 
Tom Bachmann
\ShortAuthor 
\EndTitle
\Abstract 
We show that the category of motivic spaces with transfers along finite flat morphisms, over a perfect field, satisfies all the properties we have come to expect of good categories of motives.
In particular we establish the analog of Voevodsky's cancellation theorem.
\EndAbstract
\MSC 
14F42, 19E15
\EndMSC
\KEY 
transfers, cancellation theorem, motives
\EndKEY
\Address 
LMU Munich
Mathematisches Institut
Theresienstr. 39
80333 München
Germany
\Address
\Address
\Address
\EndAddress

\setcounter{tocdepth}{1}

\section{Introduction}
\subsection{Main results}
Let $k$ be a perfect field.
Denote by $\Cor^\flf(k)$ the $(2,1)$-category with objects the smooth $k$-schemes and groupoid of morphisms the spans $X \xleftarrow{p} Z \to Y$, where $p$ is finite flat (equivalently, finite locally free).
We write $\Spc^\flf(k)$ for the motivic localization of the non-abelian derived category $\PSh_\Sigma(\Cor^\flf(k))$ of $\Cor^\flf(k)$ and $\SH^\flf(k)$ for the stabilization of $\Spc^\flf(k)$ with respect to (the image of) $\P^1$.
Our main result (see Theorem \ref{thm:main}) is that the canonical functor \[ \Spc^\flf(k)^\gp \to \SH^\flf(k) \] is fully faithful, where $\Spc^\flf(k)^\gp$ denotes the subcategory of grouplike objects.
We also show (see Theorem \ref{thm:present-kgl}) that under the induced adjunction \[ \SH(k) \adj \SH^\flf(k): \mu^* \] we have $\mu^*(\1) \wequi \kgl$, the effective algebraic $K$-theory spectrum \cite[Definition 5.5]{spitzweck2012motivic}.\footnote{I.e. the effective (or equivalently very effective) cover of the algebraic $K$-theory spectrum $\KGL$.}

\subsection{Proof overview}
Voevodsky and Suslin have provided us with a recipe for proving results of the above style \cite{voevodsky2002cancellation,suslin2003grayson} which has since been replicated several times; see e.g. \cite{FaselCancellation,bachmann-criterion,ananyevskiy2016cancellation}.
Essentially, for any sufficiently nice category of correspondences, the motivic localization functor always takes the form $L_\Zar L_{\A^1}$ and consequently can be controlled quite effectively.
Fully faithfulness of stabilization then reduces to constructing for every correspondence \[ \alpha: X \times \Gm \leftarrow Z \to Y \times \Gm \] new correspondences $\rho_n(\alpha)$ from $X$ to $Y$, satisfying certain properties.
In fact in all known cases the correspondences $\rho_n(\alpha)$ are obtained by intersecting with a family of specific cycles on $\Gm \times \Gm$; the complication that always arises is that this intersection need not be ``nice'' any more, e.g. in our case it need no longer be finite flat over $X$.
The main result of this work is Corollary \ref{cor:exhaustive}, where we show that flatness holds for $n$ sufficiently large.
Cancellation is then proved by closely following Voevodsky's original argument, just taking into account that $\Cor^\flf(k)$ is no longer a $1$-category nor locally group complete or $\Z$-linear.

The identification of $\mu^*(\1)$ proceeds in two steps.
We first show that $\mu^*(\1) \in \SH(k)^\veff$.
By an argument of Suslin, this follows from a certain property of $\Cor^\flf(k)$ called \emph{rational contractibility} (see Definitions \ref{def:C-satisfies-ratl-contr} and \ref{def:ratl-contr}).
This is again established by essentially imitating Suslin's proof of rational contractibility for finite correspondences, taking into account the small adjustments needed working with categories that are neither locally group complete, $\Z$-linear nor $1$-categories.
Once this is proved, the identification of $\mu^*(\1)$ with $\kgl$ is essentially the main theorem of \cite{hoyois2020hilbert}.

\subsection{Organization}
In \S\ref{sec:formalism} we review and extend the Voevodsky--Suslin formalism for proving properties of motives built out of sufficiently good categories of correspondences.
The most interesting result here is perhaps Proposition \ref{prop:abstract-cancel}, which reformulates Voevodsky's proof of the cancellation theorem in an abstract setting.
In \S\ref{sec:ffl} we study the category of motivic spaces with finite flat transfers, by showing that it satisfies the assumptions on a ``good'' category of correspondences exposed in the previous section.
Using this together with the formalism, we easily establish our main results in \S\ref{sec:applications}.
In Appendix \ref{app:ratl-contr} we reformulate part of Suslin's theory of rationally contractible presheaves of abelian groups for presheaves of spaces.

\subsection{Notation}\label{subsec:notation}
Unless stated otherwise, by a presheaf we mean a presheaf of spaces.
We write $\Sm_k$ for the category of smooth (qcqs) $k$-schemes, and $\EssSm_k$ for the category of essentially smooth $k$-schemes (by this we mean schemes obtained as cofiltered limits of diagrams of smooth $k$-schemes with affine \emph{étale} transition maps).
In a closed symmetric monoidal $\infty$-category, we denote by $\iMap(\ph,\ph)$ the internal mapping object.
We denote by $\Spc$ the $\infty$-category of spaces, and by $\SH$ the $\infty$-category of spectra.
Given a category $\scr C$ with finite coproducts, we denote by $\PSh_\Sigma(\scr C)$ its nonabelian derived category \cite[\S5.5.8]{lurie-htt}.

\subsection{Acknowledgments}
My understanding of the cancellation theorem comes primarily from discussions with Håkon Kolderup and the participants of the Harvard Thursday seminar, which studied motivic infinite loop space theory in fall of 2019.
I would like to heartily thank all of them.

\section{The motivic formalism} \label{sec:formalism}
\subsection{Generalities}
\subsubsection{}
Let $S$ be a scheme.
We denote by $\Cor^\fr(S)$ the symmetric monoidal $\infty$-category with objects the smooth $S$-schemes and morphism spaces given by the tangentially framed correspondences \cite[\S4]{EHKSY}.
\begin{definition} \label{def:mot-correspondences}
By a \emph{motivic category of correspondences} over $S$ we mean a symmetric monoidal functor $\mu: \Cor^\fr(S) \to \scr C$ satisfying the following conditions:
\begin{enumerate}
\item $\scr C$ is semiadditive, and its tensor product commutes with finite coproducts,
\item $\mu$ preserves finite coproducts and is essentially surjective,
\item $\mu$ is compatible with the Nisnevich topology, a notion to be explained below.
\end{enumerate}
\end{definition}
Given any functor $\mu: \Cor^\fr(S) \to \scr C$ preserving finite coproducts, there is an induced adjunction \[ \mu: \PSh_\Sigma(\Cor^\fr(S)) \adj \PSh_\Sigma(\scr C): \mu^*. \]
Recall also the adjunction \cite[\S3.2.11]{EHKSY} \[ \gamma^*: \PSh_\Sigma(\Sm_S)_* \adj \PSh_\Sigma(\Cor^\fr(S)). \]
We put \[ h^{\scr C} := \gamma_*\mu^*\mu\gamma^*: \PSh_\Sigma(\Sm_S)_* \to \PSh_\Sigma(\Sm_S)_* \quad\text{and}\quad U = \gamma_*\mu^*: \PSh_\Sigma(\scr C) \to \PSh_\Sigma(\Sm_S)_*, \] and we write \[ h_+^\scr{C} = h^\scr{C}((\ph)_+): \PSh_\Sigma(\Sm_S) \to \PSh_\Sigma(\Sm_S)_* \quad\text{and}\quad \mu_+ = \mu(\gamma^*(\ph)_+): \PSh_\Sigma(\Sm_S) \to \PSh_\Sigma(\scr C). \]
When no confusion can arise, we denote also by $\mu$ the composite functors \[ \PSh_\Sigma(\Sm_S)_* \to \PSh_\Sigma(\scr C), \Sm_{S*} \to \scr C, \] by $\mu_+$ the composite functor $\Sm_S \to \scr C$, and so on.
\begin{definition}
Let $U \to X \in \Sm_S$ be a Nisnevich covering and write $R \hookrightarrow X$ for its associated sieve.
Then $R \in \PSh_\Sigma(\Sm_S)$ and we say that $\mu$ \emph{is compatible with the Nisnevich topology} if $h^{\scr C}_+(R) \to h^{\scr C}_+(X) \in \PSh_\Sigma(\Sm_S)_*$ is a Nisnevich equivalence.
\end{definition}

\begin{example}
$\scr C = \Cor^\fr(S)$ and $\mu=\id$ defines a motivic category of correspondences, by \cite[\S3.2]{EHKSY}.
\end{example}

\subsubsection{}
Now let $(\mu, \scr C)$ be a motivic category of correspondences.
Day convolution turns $\PSh_\Sigma(\scr C)$ into a presentably symmetric monoidal category and $\mu: \PSh_\Sigma(\Cor^\fr(S)) \to \PSh_\Sigma(\scr C)$ into a symmetric monoidal functor \cite[Proposition 4.8.1.10]{lurie-ha}.

\subsubsection{}
We call $F \in \PSh_\Sigma(\scr C)$ $\A^1$-local or Nisnevich local if the same is true for the restriction of $F$ to $\Sm_S$, and motivically local if it is both $\A^1$-local and Nisnevich local.
We write $L_{\A^1}$, $L_\Nis$ and $L_\mot$ for the associated localization functors.

\begin{lemma} \label{lemm:C-mot-loc}
The forgetful functors $\mu^*: \PSh_\Sigma(\scr C) \to \PSh_\Sigma(\Cor^\fr(S))$ and $U: \PSh_\Sigma(\scr C) \to \PSh_\Sigma(\Sm_S)_*$ commute with $L_{\A^1}, L_\Nis$ and $L_\mot$.
\end{lemma}
\begin{proof}
By construction, the forgetful functors preserve and detect local objects and equivalences; it hence suffices to show that they preserve weak equivalences.
It is enough to  prove the second claim, since $\gamma_*$ preserves and detects weak equivalences by \cite[Proposition 3.2.14]{EHKSY}.
Using \cite[Lemma 2.10]{bachmann-norms}, we reduce to proving that $h^{\scr C}$ preserves $\A^1$-homotopy equivalences and generating Nisnevich equivalences.
The first case is formal (see e.g. the proof of \cite[Lemma 2.3.20]{EHKSY}) and the second case follows from our assumption of compatibility with the topology, since if $R \hookrightarrow X$ is the Nisnevich sieve generated by a covering $\{U_i \to X\}$ then $L_\Sigma R$ is the sieve generated by $\coprod_i U_i \to X$.
\end{proof}

\begin{lemma} \label{lemm:C-semiadd}
The category $\PSh_\Sigma(\scr C)$ is semiadditive and $\mu^*$ preserves all colimits.
\end{lemma}
\begin{proof}
Since $\scr C$ is semiadditive by assumption, the same holds for $\PSh_\Sigma(\scr C)$ by \cite[Corollary 2.4]{gepner2016universality}.
The functor $\mu^*$ preserves sifted colimits essentially by construction \cite[Proposition 5.5.8.10(4)]{lurie-htt}, and finite products being a right adjoint.
Hence it preserves finite coproducts since source and target are semiadditive, and thus all colimits by \cite[Lemma 2.8]{bachmann-norms}.
\end{proof}

We write $\Spc^\scr{C}$ for the motivic localization of $\PSh_\Sigma(\scr C)$.
\begin{proposition}
\begin{enumerate}
\item $\Spc^\scr{C}$ is presentably symmetric monoidal and compactly generated under sifted colimits by the images of smooth schemes.
\item $\Spc^\scr{C}$ is semiadditive.
\item There is an induced adjunction \[ \mu: \Spc^\fr(S) \adj \Spc^\scr{C}: \mu^* \] with $\mu$ symmetric monoidal and $\mu^*$ cocontinuous and conservative.
\end{enumerate}
\end{proposition}
\begin{proof}
Immediate from Lemmas \ref{lemm:C-mot-loc} and \ref{lemm:C-semiadd}, see e.g. \cite[proofs of Propositions 3.2.10 and 3.2.15]{EHKSY}.
\end{proof}

\subsubsection{}
Being semiadditive, $\Spc^\scr{C}$ has a subcategory of grouplike objects which we denote by $\Spc^{\scr{C}\gp} \subset \Spc^{\scr C}$.
We write \[ (\ph)^\gp: \Spc^{\scr C} \to \Spc^{\scr{C}\gp} \] for the reflection.
Note that this is a symmetric monoidal localization.
Note also that $F \in \Spc^\scr{C} \subset \PSh_\Sigma(\scr C)$ is grouplike if and only if $F(X)$ is grouplike for every $X \in \Sm_S$.
Note finally that a commutative monoid $M$ is grouplike if and only if the shearing map $M \times M \to M \times M$ is an equivalence; this condition is clearly stable under limits and sifted colimits, and hence under arbitrary colimits of monoids.\todo{really?}

\subsubsection{}
Put $\SHSC = \Spc^{\scr C}[(S^1)^{-1}]$ and $\SH^\scr{C} = \Spc^{\scr C}[(\P^1)^{-1}]$.
(Here for a presentably symmetric monoidal $\infty$-category $\scr C$ and an object $X \in \scr C$, we denote by $\scr C[X^{-1}]$ the initial presentably symmetric monoidal $\infty$-category under $\scr C$ in which $X$ becomes invertible \cite[\S2.1]{robalo}.)
We thus get a commutative diagram of left adjoints
\begin{equation*}
\begin{tikzcd}
\Spc(S)_* \ar[r, "\gamma^*"] \ar[d, "\Sigma^\infty_{S^1}"] & \Spc^\fr(S) \ar[r, "\mu"] \ar[d, "\Sigma^\infty_{S^1}"] & \Spc^\scr{C} \ar[d, "\Sigma^\infty_{S^1}"] \ar[dd, "\Sigma^\infty", bend left=50] \\
\SHS(S) \ar[r, "\gamma^*"] \ar[d, "\sigma^\infty"] & \SHSfr(S) \ar[r, "\mu"] \ar[d, "\sigma^\infty"] & \SHSC \ar[d, "\sigma^\infty"] \\
\SH(S) \ar[r, "\gamma^*"]  & \SH^\fr(S) \ar[r, "\mu"]  & \SH^\scr{C}.
\end{tikzcd}
\end{equation*}
The right adjoints of $\Sigma^\infty, \Sigma^\infty_{S^1}, \sigma^\infty$ are respectively denoted $\Omega^\infty, \Omega^\infty_{S^1}, \omega^\infty$.
Recall that $\gamma^*: \SH(S) \to \SH^\fr(S)$ is an equivalence \cite[Theorem 18]{hoyois2018localization}.

\subsection{The case of perfect fields}
From now on we assume that $S = Spec(k)$ is the spectrum of a perfect field.
\begin{proposition} \label{prop:mot-loc-perf}
Let $F \in \PSh_\Sigma(\scr C)^\gp$.
Then the canonical map \[ L_{\A^1} F \to L_\mot F \] induces an equivalence on sections over essentially smooth\footnote{Recall our conventions regarding essentially smooth schemes from \S\ref{subsec:notation}.}, semilocal $k$-schemes.
In particular \[ L_\mot F \wequi L_\Nis L_{\A^1} F \] and \[ UL_\mot F \wequi L_\Zar L_{\A^1} U F. \]
\end{proposition}
\begin{proof}
Since $\mu^*$ commutes with $L_\Nis, L_{\A^1}, L_\mot$ (Lemma \ref{lemm:C-mot-loc}), we are reduced to the case $\scr C = \Cor^\fr(k)$, which is treated in \cite[Theorem 3.4.11]{EHKSY}.
\end{proof}

\begin{proposition} \label{prop:Omega-Gm-mot-loc}
Let $F \in \PSh_\Sigma(\scr C)^\gp$.
Then the canonical map \[ L_\mot \Omega_\Gm F \to \Omega_\Gm L_\mot F \] is an equivalence.
\end{proposition}
\begin{proof}
$\Omega_\Gm$ preserves grouplike objects and commutes with $L_{\A^1}$ (see e.g. \cite[Lemma 4]{bachmann-grassmann}).
It thus suffices (using Proposition \ref{prop:mot-loc-perf}) to prove that if $F$ is grouplike and $\A^1$-local, then $L_\Zar \Omega_\Gm UF \to \Omega_\Gm L_\Zar UF$ is an equivalence.
Using hypercompleteness (see e.g. \cite[Proposition A.3]{bachmann-norms}), we may check this on homotopy sheaves.
Since homotopy sheaves are strictly $\A^1$-invariant \cite[Corollary 3.4.13]{EHKSY} and hence unramified \cite[Definition 2.1]{A1-alg-top}\cite[Lemma 6.4.4]{morel2005stable}, we need only check that we have an isomorphism on generic stalks, i.e. finitely generated fields $K/k$.
It thus suffices to show that \[ H^i_\Zar(\Spec(K)_+ \wedge \Gm, \pi_j F) \wequi \begin{cases} \pi_j(F)(\Spec(K)_+ \wedge \Gm) & i=0 \\ 0 & \text{else} \end{cases}. \]
Here by $\pi_j F$ we mean the homotopy \emph{presheaf} of $F$.
Since $\Sigma \Spec(K)_+ \wedge \Gm \stackrel{\A^1}{\wequi} \P^1 \wedge \Spec(K)_+$ and $L_\Zar \pi_j F$ is motivically local (by Proposition \ref{prop:mot-loc-perf}), we have \[ H^i_\Zar(\Spec(K)_+ \wedge \Gm, \pi_j F) \wequi H^{i+1}_\Zar(\P^1 \wedge \Spec(K)_+,  \pi_j F) = 0 \text{ for } i>0, \] for cohomological dimension reasons.
It remains to prove the first isomorphism, for which it suffices to show that the restriction of $\pi_j F$ to $\A^1_K$ is already a Zariski sheaf.
If $K$ is infinite we can, arguing as in \cite[proof of Theorem 3.4.11]{EHKSY}, refer to \cite[Theorem 2.15(2)]{garkusha2015homotopy} for this.

The theorem is thus proved if $k$ is infinite.
Since $L_\mot$ and $\Omega_\Gm$ commute with essentially smooth base change, we may reduce to this case using \cite[Corollary B.2.5]{EHKSY}.
\end{proof}
\begin{corollary} \label{cor:Omega-Gm-colimits}
The functor $\Omega_\Gm: \Spc^{\scr C\gp} \to \Spc^{\scr C\gp}$ preserves colimits.
\end{corollary}
\begin{proof}
The functor preserves finite products, whence by semiadditivity we need only prove that it preserves sifted colimits \cite[Lemma 2.8]{bachmann-norms}.
Via Proposition \ref{prop:Omega-Gm-mot-loc}, we are reduced to proving that $\Omega_\Gm$ commutes with sifted colimits on $\PSh_\Sigma(\scr C)^\gp$.
Since sifted colimits in this category are computed sectionwise, this is clear for $\Omega_{\Gm_+}$.
Since colimits are stable under retracts\NB{which are themselves colimits...}, the result follows.
\end{proof}

\subsection{Cancellation}
We still assume that $S$ is the spectrum of a perfect field.
\begin{lemma} \label{lemm:S1-cancel}
The transformation $\id \to \Omega\Sigma$ of endofunctors of $\Spc^{\scr C\gp}$ is an equivalence.
\end{lemma}
\begin{proof}
Since $\mu^*$ preserves limits (being a right adjoint), and also colimits (by Lemma \ref{lemm:C-semiadd}), it preserves the final object $*$ and thus commutes with $\Omega$ and $\Sigma$.
We may thus reduce to $\scr C = \Cor^\fr$, in which case the result is an immediate consequence of \cite[Corollary 3.5.6]{EHKSY}.
\end{proof}

\begin{definition}
We say that $\scr C$ \emph{satisfies cancellation} if for every $X \in \Sm_k$ the canonical map \[ \mu_+(X)^\gp \to \Omega_\Gm \Sigma_\Gm \mu_+(X)^\gp \in \Spc^{\scr C\gp} \] is an equivalence.
\end{definition}

\begin{lemma} \label{lemm:Gm-cancel-deduction}
Suppose that $\scr C$ satisfies cancellation.
Then the transformations $\id \to \Omega_\Gm\Sigma_\Gm$ and $\id \to \Omega_{\P^1}\Sigma_{\P^1}$ of $\Spc^{\scr C \gp}$ are equivalences.
\end{lemma}
\begin{proof}
We first treat $\Omega_\Gm\Sigma_\Gm$.
Since the functor preserves colimits by Corollary \ref{cor:Omega-Gm-colimits}, we need only show that the transformation induces an equivalence on generators, which holds by assumption.
Now we treat $\Omega_{\P^1} \Sigma_{\P^1}$.
Since $\P^1 \wequi S^1 \wedge \Gm$ we get $\Omega_{\P^1} \wequi \Omega_{S^1}\Omega_{\Gm}$ and $\Sigma_{\P^1} \wequi \Sigma_\Gm\Sigma_{S^1}$; thus the claim reduces via lemma \ref{lemm:S1-cancel} to the one about $\Omega_\Gm\Sigma_\Gm$.
\end{proof}

\begin{proposition} \label{prop:cancellation}
If $\scr C$ satisfies cancellation, then the functors \[ \Spc^{\scr C\gp} \to \SHSC \to \SH^{\scr C} \] are fully faithful.
\end{proposition}
\begin{proof}
If $\scr D$ is a compactly generated, presentably symmetric monoidal $\infty$-category, $T \in \scr D$ a compact symmetric object, \[ \Sigma^\infty: \scr D \adj \scr D[T^{-1}]: \Omega^\infty \] is the stabilization adjunction, $X \in \scr D$ compact and $E \in \scr D$ arbitrary, then \cite[p. 467]{hoyois2016equivariant} \[ \Map(X, \Omega^\infty \Sigma^\infty E) \wequi \colim_n \Map(\Sigma_T^n X, \Sigma_T^n E). \]
In particular, if $\id \to \Omega_T \Sigma_T$ is an equivalence, then $\scr D \to \scr D[T^{-1}]$ is fully faithful.

We first apply this with $\scr D = \Spc^{\scr C \gp}$ and $T = S^1$.
The assumptions are satisfied by Lemma \ref{lemm:S1-cancel} and hence $\Spc^{\scr C\gp} \to \SHSC$ is fully faithful.
Next we apply this with $\scr D = \Spc^{\scr C \gp}$ and $T = \P^1$.
The assumptions are satisfied by Lemma \ref{lemm:Gm-cancel-deduction} and \cite[Lemma 3.3.3]{EHKSY}.
Now let $X, Y \in \Spc^{\scr C \gp}$ and $N + i \ge 0, N \ge 0$.
Then \begin{gather*} \Map(\sigma^\infty \Sigma^i \Sigma^\infty_{S^1} X, \sigma^\infty \Sigma^\infty_{S^1} Y) \wequi \Map(\Sigma^\infty \Sigma^{i+N} X, \Sigma^\infty \Sigma^N Y) \wequi \Map(\Sigma^{i+N} X, \Sigma^N Y) \\ \wequi \Map(\Sigma^\infty_{S^1} \Sigma^{i+N} X, \Sigma^\infty_{S^1} \Sigma^N Y) \wequi \Map(\Sigma^{i} \Sigma^\infty_{S^1} X, \Sigma^\infty_{S^1} Y), \end{gather*} or in other words \[ \Map(\Sigma^i \Sigma^\infty_{S^1} X, \omega^\infty \sigma^\infty \Sigma^\infty_{S^1} Y) \wequi \Map(\Sigma^i \Sigma^\infty_{S^1} X, \Sigma^\infty_{S^1} Y). \]
Since $\SHSC$ is generated under colimits by objects of the form $\Sigma^i \Sigma^\infty_{S^1} X$ we deduce that $\omega^\infty\sigma^\infty \Sigma^\infty_{S^1} Y \wequi \Sigma^\infty_{S^1} Y$.
Since $\omega^\infty\sigma^\infty$ preserves colimits and desuspensions (by stability and compact generation) and $\SHSC$ is generated under colimits and desuspensions by objects of the form $\Sigma^\infty Y$, we deduce that $\omega^\infty\sigma^\infty \wequi \id$, as needed.
\end{proof}

\begin{lemma} \label{lemm:connectivity} \NB{true over any base}
Let $X \in \Spc^{\scr C}$.
Then $\Sigma X \in \Spc^{\scr C\gp}$.
\end{lemma}
\begin{proof}
Grouplike objects are preserved by $L_\Nis$ and $L_{\A^1}$\NB{details?} and hence $L_\mot$, whence it suffices to prove the analogous claim for $\PSh_\Sigma(\scr C)$.
In this case suspension is computed sectionwise when viewed as taking values in monoids (see the proof of Lemma \ref{lemm:S1-cancel}).
We are reduced to the well-known observation that if $Y \in \CMon(\Spc)$ then $\Sigma Y$ ($=BY$) is grouplike (indeed connected\NB{ref?}).
\end{proof}

\begin{corollary} \label{cor:representing-spectrum}
Assume that $\scr C$ satisfies cancellation.
Consider the adjunction \[ \mu: \SH(k) \wequi \SH^\fr(k) \adj \SH^{\scr C}: \mu^*. \]
Then \[ \mu^*(\1_\scr{C}) \wequi (L_\mot h^{\scr C}_+(*)^\gp, L_\mot h^{\scr C}(\P^1), L_\mot h^{\scr C}((\P^1)^{\wedge 2}), \dots) \in \SH(k) \] is the presentation of $\mu^*(\1_{\scr C})$ as a motivic $\P^1$-$\Omega$-spectrum.
\end{corollary}
\begin{proof}
We need to determine $\Omega^\infty(\mu^*(\1_\scr{C}) \wedge (\P^1)^{\wedge n})$.
Since $(\P^1)^{\wedge n}$ is invertible, we have $\mu^*(\ph \wedge (\P^1)^{\wedge n}) \wequi \mu^*(\ph) \wedge (\P^1)^{\wedge n}$.\footnote{We use the following well-known fact: given an adjunction $F: \scr C \adj \scr D: G$ with $F$ symmetric monoidal, $X \in \scr C$ strongly dualizable and $Y \in \scr D$, then the canonical map $X \otimes GY \to G(FX \otimes Y)$ is an equivalence.}\NB{ref?}
Since $\mu^*$ commutes with $\Omega^\infty$, we get using Lemmas \ref{lemm:Gm-cancel-deduction} and \ref{lemm:C-mot-loc} that \[ \Omega^\infty(\mu^*(\1) \wedge (\P^1)^{\wedge n}) \wequi \Omega^\infty \mu^*(\Sigma^\infty \mu((\P^1)^{\wedge n})) \wequi \mu^*\Omega^\infty\Sigma^\infty\mu((\P^1)^{\wedge n}) \wequi \mu^*\mu((\P^1)^{\wedge n})^\gp \wequi L_\mot h^{\scr C}((\P^1)^{\wedge n})^\gp. \]
It remains to prove that for $n \ge 1$ we have \[ L_\mot h^{\scr C}((\P^1)^{\wedge n})^\gp \wequi L_\mot h^{\scr C}((\P^1)^{\wedge n}); \] since $\P^1 \wequi \Sigma \Gm$ this is immediate from Lemma \ref{lemm:connectivity}.
\end{proof}

\subsection{Proving cancellation}
\begin{proposition} \label{prop:abstract-cancel}
Let $\scr C$ be an additive, symmetric monoidal, ordinary $1$-category, $G \in \scr C$ and $\1 \xrightarrow{i} G \xrightarrow{p} \1$ a retraction which admits a complement $\bar G$.
Assume that $\Sigma_G := G \otimes \ph$ admits a right adjoint $\Omega_G$.
Suppose given a set $\scr S$ of objects of $\scr C$ with $\1 \in \scr S$ and closed under tensor products with $G$, as well as for each $X \in \scr S$ a map \[ \rho_X: \Omega_G \Sigma_G X \to X. \]
Assume that the following hold.
\begin{enumerate}
\item For $X \in \scr S$, the following diagram commutes
\begin{equation*}
\begin{CD}
G \otimes \Omega_G \Sigma_G X @>{G \otimes \rho_X}>> G \otimes X \\
@VVV                                                   @|        \\
\Omega_G \Sigma_G (G \otimes X) @>{\rho_{G \otimes X}}>> G \otimes X
\end{CD}
\end{equation*}
  where the left hand vertical map is an instance of the general type of map $A \otimes \Omega_G \Sigma_G X \to \Omega_G \Sigma_G (A \otimes X)$ (adjoint to $\Sigma_G A \otimes \Omega_G \Sigma_G X \wequi A \otimes \Sigma_G \Omega_G \Sigma_G X \to A \otimes \Sigma_G X \wequi \Sigma_G A \otimes X$).

\item For $X \in \scr S$, the following diagram commutes
\begin{equation*}
\begin{CD}
\Omega_G \Sigma_G (X) @>{\rho_{X}}>> X \\
@V{\Omega_G \Sigma_G i \otimes X}VV         @V{i \otimes X}VV \\
\Omega_G \Sigma_G (G \otimes X) @>{\rho_{G \otimes X}}>> G \otimes X \\
@V{\Omega_G \Sigma_G p \otimes X}VV         @V{p \otimes X}VV \\
\Omega_G \Sigma_G (X) @>{\rho_{X}}>> X. \\
\end{CD}
\end{equation*}

\item For $X \in \scr S$, the composite \[ X \xrightarrow{\bar u} \Omega_{\bar G} \Sigma_{\bar G} X \subset \Omega_G \Sigma_G X \xrightarrow{\rho_X} X \] is the identity.
\item The object $\bar G$ is symmetric.\footnote{I.e. the cyclic permutation on ${\bar G}^{\otimes n}$ is the identity for some $n \ge 2$, in which case $\bar G$ is called $n$-symmetric.}
\end{enumerate}

Then for all $X \in \scr S$ the unit map $X \xrightarrow{\bar u} \Omega_{\bar G} \Sigma_{\bar G} X$ is an equivalence.
\end{proposition}
\begin{proof}\NB{this is so fishy!!!}
Write $\bar \rho$ for the composite $\Omega_{\bar G} \Sigma_{\bar G} X \subset \Omega_G \Sigma_G X \xrightarrow{\rho_X} X$.
By assumption, $\bar \rho \bar u = \id_X$.
It remains to prove that $\bar u \bar \rho = \id_{\Omega_{\bar G}\Sigma_{\bar G} X}$.

As a warm-up, note that there are \emph{two} natural maps $\Omega_{\bar G} \Sigma_{\bar G} X \to \Omega_{\bar G}^2 \Sigma_{\bar G}^2 X$, corresponding to ``inserting an identity on the left or right factor''.
The map $\bar u_1 = \Omega_{\bar G} u_{\Sigma_{\bar G} X}$ inserts the identity on the left, and the map $\bar u_1'$ inserting the identity on the right is obtained by conjugating with the twist map.
More generally we have a map \[ \bar u_n = \Omega_{\bar G}^n u_{\Sigma_{\bar G}^n X}: \Omega_{\bar G}^n \Sigma_{\bar G}^n X \to \Omega_{\bar G}^{n+1} \Sigma_{\bar G}^{n+1} X \] inserting an identity on the left, and other natural maps with the same source and target are obtained by acting with the symmetric group $S_{n+1}$.

Write $\Sigma^n_G X = \Sigma^n_{\bar G} X \oplus R$.
We can think of $\rho_{\Sigma^n_G X}$ as a two by two matrix, with upper left hand corner a map \[ \rho_{\Sigma^n_{\bar G} X}: \Omega_G \Sigma_G \Sigma^n_{\bar G} X \to \Sigma^n_{\bar G} X. \]
Similarly $\bar u_{\Sigma^n_G X}$ can be though of as a two by two matrix, but this time a diagonal one (by naturality).
Since $\rho_{\Sigma^n_G X} \bar u_{\Sigma^n_G X}$ is the identity, we deduce that also $\rho_{\Sigma^n_{\bar G} X} \bar u_{\Sigma^n_{\bar G} X}$ must be the identity.
In other words, (3) also holds for objects of the form $\Sigma^n_{\bar G} X$.

Define a map \[ \bar \rho_n: \Omega_{\bar G}^{n+1} \Sigma_{\bar G}^{n+1} X \subset \Omega_{\bar G}^n \Omega_G \Sigma_G \Sigma_{\bar G}^n X \xrightarrow{\rho_{\Sigma_{\bar G}^n X}} \Omega_{\bar G}^n \Sigma_{\bar G}^n X. \]
Then \begin{equation}\label{eq:rho-u-induced} \bar \rho_n \bar u_n = \id; \end{equation} indeed this map is obtained by applying $\Omega_{\bar G}^n$ to the composite \[ \Sigma_{\bar G}^n X \xrightarrow{\bar u} \Omega_{\bar G} \Sigma_{\bar G} (\Sigma_{\bar G}^n X) \xrightarrow{\rho_{\Sigma^n_{\bar G} X}} \Sigma_{\bar G}^n X, \] which is the identity by the extended version of (3) that we just established.

For $N \ge 2$ and $\sigma \in S_N$ acting on ${\bar G}^{\otimes N}$, consider the composite \[ p(\sigma): \Omega_{\bar G} \Sigma_{\bar G} X \xrightarrow{\bar u_1}  \Omega_{\bar G}^2 \Sigma_{\bar G}^2 X \xrightarrow{\bar u_2} \dots \to \Omega_{\bar G}^N \Sigma_{\bar G}^N X \xrightarrow{\sigma} \Omega_{\bar G}^N \Sigma_{\bar G}^N X \xrightarrow{\bar\rho_{N-1}} \Omega_{\bar G}^{N-1} \Sigma_{\bar G}^{N-1} X \to \dots \xrightarrow{\bar \rho_1} \Omega_{\bar G} \Sigma_{\bar G} X. \]
If $\sigma = \id$ then by iterated application of \eqref{eq:rho-u-induced} we get \[ p(\id) = \id. \]
On the other hand if $\sigma=\sigma_l$ is the permutation ``cycling to the left'' we find that (suppressing subscripts for readability) \[ \sigma_l \circ \bar u \circ \dots \circ \bar u = \bar u \circ \dots \circ \bar u \circ \bar u_1', \] i.e. ``one of the identities has been inserted at the right''.
Consequently $p(\sigma_l) = \bar \rho_1 \bar u_1'$ (again using \eqref{eq:rho-u-induced} repeatedly).
If furthermore $\bar G$ is $N$-symmetric then we obtain \[ \id_{\Sigma_{\bar G} \Omega_{\bar G}} X = p(\id) = p(\sigma_l) = \bar \rho_1 \bar u_1'. \]
It thus suffices to show that the following diagram commutes
\begin{equation*}
\begin{CD}
\Omega_{\bar G} \Sigma_{\bar G} X @>{\bar u_1'}>> \Omega_{\bar G}^2 \Sigma_{\bar G}^2 X \\
@V{\bar \rho}VV                                    @V{\bar \rho_1}VV \\
X                                 @>{\bar u}>> \Omega_{\bar G} \Sigma_{\bar G} X.
\end{CD}
\end{equation*}
By adjunction and definition of $\bar \rho$, this is equivalent to the commutativity the outer square in the following diagram
\begin{equation*}
\begin{CD}
\bar G \otimes \Omega_{\bar G} \Sigma_{\bar G} X @>>> \Omega_{\bar G} \Sigma_{\bar G} ({\bar G} \otimes X) \\
@VVV                                                      @VVV \\
\bar G \otimes \Omega_{G} \Sigma_{G} X @>>> \Omega_{G} \Sigma_{G} ({\bar G} \otimes X) \\
@V{\id_{\bar G} \otimes \rho_X}VV                         @V{\rho_{\Sigma_{\bar G} X}}VV  \\
\bar G \otimes X                                 @= \bar G \otimes X.
\end{CD}
\end{equation*}
The top square commutes without any assumptions.
For the bottom square, it will suffice to prove that $\rho_{\Sigma_G X}$ respects the decomposition $\Sigma_G X \wequi \Sigma_{\bar G} X \oplus X$.
Indeed then we may replace all instances of $\bar G$ by $G$, and hence have commutativity by (1).
We can write $\rho_{\Sigma_G X}$ as the matrix $\begin{pmatrix} \rho_{\Sigma \bar G} & f \\ g & h \end{pmatrix}$.
Then in condition (2) the upper square implies that $f=0$ (and $h=\rho_X$), whereas the lower square implies that $g=0$ (and $h=\rho_X$).
The result follows.
\end{proof}

\subsection{Rational contractibility}
We recall the notion of a rationally contractible presheaf in Appendix \ref{app:ratl-contr}.
\begin{definition} \label{def:C-satisfies-ratl-contr}
We say that $\scr C$ \emph{satisfies rational contractibility} if for every $n>0$ the presheaf $h^{\scr C}(\Gmp{n})^\gp$ is rationally contractible.
\end{definition}

\begin{proposition} \label{prop:ratl-contr-veff}
Suppose that $\scr C$ satisfies cancellation and rational contractibility.
Then \[ \mu^*(\1_{\SH^{\scr C}}) \in \SH(k)^\veff. \]
\end{proposition}
\begin{proof}
By \cite[Theorem 4.4]{bachmann-criterion}, in order to show that $\mu^*(\1) \in \SH(k)^\eff$ we need to show that for all $n>0$ and $K/k$ finitely generated we have \[ |\omega^\infty(\mu^*(\1) \wedge \Gmp{n})(\hat\Delta^\bullet_K)| \wequi 0. \]
Similarly by \cite[\S3]{bachmann-very-effective}, assuming that $\mu^*(\1) \in \SH(k)^\eff$, in order to show that $\mu^*(\1) \in \SH(k)^\veff$ we need to show that $\omega^\infty(\mu^*(\1))$ has vanishing negative homotopy sheaves.
For this it suffices to show that \begin{equation}\label{eq:veff-interm} |\Omega^\infty(\mu^*(\1) \wedge \Gmp{n})(\hat\Delta^\bullet_K)| \wequi 0 \quad\text{and}\quad \Omega^\infty(\Sigma^i \mu^*(\1) \wedge \Gmp{n}) \in \Spc(k)_{\ge i} \end{equation} for all $i \ge 0$.
Indeed the second condition implies that $(\mu^*(\1) \wedge \Gmp{n})(\hat\Delta^\bullet_K)$ takes values in $\SH_{\ge 0}$ (using Proposition \ref{prop:mot-loc-perf}, say), and then $\Omega^\infty$ commutes with the geometric realization \cite[Proposition 1.4.3.8]{lurie-ha}.

As in Corollary \ref{cor:representing-spectrum} we have $\Sigma^i \mu^*(\1) \wedge \Gmp{n} \wequi \mu^*(\Sigma^\infty \mu(\Sigma^i\Gmp{n}))$ and hence \[ \Omega^\infty(\Sigma^i \mu^*(\1) \wedge \Gmp{n}) \wequi L_\mot h^\scr{C}(\Sigma^i \Gmp{n})^\gp; \] here we have used Lemma \ref{lemm:Gm-cancel-deduction}.
The functor $h^\scr{C}$ commutes with $\Sigma^i$, when viewed as taking values in presheaves of commutative monoids.
By Proposition \ref{prop:mot-loc-perf} we have $L_\mot = L_\Nis L_{\A^1}$, which takes sectionwise $n$-connected presheaves to Nisnevich locally $n$-connected presheaves.
This implies the second condition in \eqref{eq:veff-interm}.
Again by Proposition \ref{prop:mot-loc-perf}, we have $L_\mot \wequi L_{\A^1}$ when evaluated on semilocal schemes; thus \[ |\Omega^\infty(\mu^*(\1) \wedge \Gmp{n})(\hat\Delta^\bullet_K)| \wequi |(L_{\A^1}h^\scr{C}(\Gmp{n})^\gp)(\hat\Delta^\bullet_K)|. \]
The first condition in \eqref{eq:veff-interm} thus follows from Corollary \ref{cor:ratl-contr-crit}.
\end{proof}

\begin{remark} \label{rmk:ratl-contr-simplified}
In order to prove that $h^\scr{C}(\Gmp{n})^\gp$ is rationally contractible, it suffices to prove that $h^\scr{C}((\A^1 \setminus 0)^{\times n},x_0)$ is rationally contractible as a presheaf of commutative monoids, where $x_0 = (1,1, \dots, 1)$ is the base point.
Indeed this implies that $h^\scr{C}((\A^1 \setminus 0)^{\times n},x_0)^\gp$ is rationally contractible by Lemma \ref{lemm:ratl-contr-gp}, and $h^\scr{C}(\Gmp{n})^\gp$ is a retract of $h^\scr{C}((\A^1 \setminus 0)^{\times n},x_0)^\gp$, so we conclude by Example \ref{ex:retract-contr}.
\end{remark}

\begin{corollary} \label{cor:identify-unit}
Suppose that $\scr C$ satisfies cancellation and rational contractibility.
Then \[ \mu^*(\1_{\SH^{\scr C}}) \wequi \Sigma^\infty_\fr \mu^*(\mu_+(*)). \]
\end{corollary}
\begin{proof}
Since $\mu^*(\1_{\SH^{\scr C}})$ is very effective and $\Spc^\fr(k)^\gp \wequi \SH(k)^\veff$ \cite[Theorem 3.5.14]{EHKSY}, we deduce that \[ \mu^*(\1_{\SH^{\scr C}}) \wequi \Sigma^\infty_\fr \Omega^\infty_\fr \mu^* \Sigma^\infty \mu_+(*) \wequi \Sigma^\infty_\fr \mu^* \Omega^\infty \Sigma^\infty \mu_+(*) \wequi \Sigma^\infty_\fr \mu^* (\mu_+(*)^\gp) \wequi \Sigma^\infty_\fr \mu^* (\mu_+(*)), \] using Lemma \ref{lemm:Gm-cancel-deduction} as well as that $\mu^*$ commutes with group completion and $\Sigma^\infty_\fr$ inverts group completion.
\end{proof}

\section{The case of finite flat correspondences} \label{sec:ffl}
\subsection{Generalities}
Let $S$ be a scheme.
Write $\Cor^\flf(S)$ for the symmetric monoidal, semiadditive $(2,1)$-category with the same objects as $\Sm_S$ and morphisms the groupoids of spans \[ X \xleftarrow{p} Z \to Y, \] where $p$ is required to be finite locally free (see e.g. \cite[\S5]{BarwickMackey} for a construction of span categories in a much more general context).
Denote by $\Cor^\fsyn(S)$ the category obtained in a similar way, but requiring $p$ to be finite syntomic.
There is an evident symmetric monoidal functor $\Cor^\fsyn(S) \to \Cor^\flf(S)$ which preserves finite coproducts.
Since a functor $\Cor^\fr(S) \to \Cor^\fsyn(S)$ was constructed in \cite[\S4.2.37]{EHKSY}, we all in all obtain a functor $\mu: \Cor^\fr(S) \to \Cor^\flf(S)$.
This satisfies all the axioms of a motivic category of correspondences (see Definition \ref{def:mot-correspondences}).
The only non-trivial part is the following.
\begin{lemma}
The category $\Cor^\flf(S)$ is compatible with the Nisnevich topology.
\end{lemma}
\begin{proof}
Let $U \to X \in \Sm_S$ be a Nisnevich covering with one element, and $R$ the associated sieve.
Then $\hflfp{R} \hookrightarrow \hflfp{X}$ consists of those spans \[ Y \leftarrow Z \xrightarrow{q} X \] where $q$ factors through $U$.
(Indeed by universality of colimits of spaces and since $h_+$ preserves sifted colimits, the fiber over $Y \leftarrow Z \xrightarrow{q} X$ is given by $|U^{\times_X \bullet} \times_X Z|$, which is either empty or contractible depending on whether $q$ factors through $U$.)
In particular this is a map of $1$-truncated presheaves, which is a Nisnevich equivalence if and only if it induces an equivalence on stalks \cite[Lemma 6.5.2.9]{lurie-htt}; i.e. we may assume that $Y$ is Nisnevich local and need to show that every $q$ factors through $U$.
But now $Z$ is a finite disjoint union of Nisnevich local schemes \cite[Tag 04GH(1)]{stacks-project}, so this is clear.
\end{proof}

\subsection{Cancellation}
\begin{lemma} \label{lemm:flatness}
Let $X$ be smooth over a perfect field $k$ and $(p, t): Z \to X \times \Gm$ finite locally free.
Let $f \in \scr O(Z)$ and let \[ Z_n = Z(1 - t^n f) \subset Z. \]
Then for $N$ sufficiently large and all $n>N$, the canonical map $Z_n \to X$ is flat.
\end{lemma}
\begin{proof}
We first show that there is a Nisnevich covering family $\{U_i \to X\}_{i \in I}$ such that $Z_{U_i} \to U_i \times \Gm$ is finite free.\NB{This is mainly where we use the assumptions on $X$. Feels like it should be true more generally?}
It suffices to show that if $L$ is the henselization of a smooth $k$-variety in a point, then all vector bundles of rank $d$ on $L \times \Gm$ are trivial (for any $d$).
We use the existence of a motivic space $\Gr_d \in \Spc(k)_*$ such that for $Y \in \Sm_k$ affine $[Y_+, \Gr_d]$ is the set of vector bundles on $Y$ up to isomorphism \cite[Theorem 5.2.3]{asok2015affine-I}.
The cofiber sequence \[ L_+ \wequi L_+ \wedge S^0 \to L_+ \wedge \Gm_+ \wequi (L \times \Gm)_+ \to L_+ \wedge \Gm \] induces a fiber sequence \[ \Map(L_+, \Omega_\Gm \Gr_d) \to \Map((L \times \Gm)_+, \Gr_d) \to \Map(L_+, \Gr_d). \]
We seek to show that the middle space is connected.
The motivic space $\Gr_d$ is connected (vector bundles being locally trivial), and hence so is $\Omega_\Gm \Gr_d$ by \cite[Theorem 6.13]{A1-alg-top}.
Thus the two outer spaces are connected, and the claim follows.

Since being flat is fpqc local on the base \cite[Tag 02L2]{stacks-project}, we may replace $X$ by $U_i$ and so assume that $Z \to X \times \Gm$ is finite free.
We may further assume that $X = \Spec(A)$ is affine.
Write $Z = \Spec(B)$.
Then $B$ is an $A$-algebra, we are provided with elements $f,t \in B$, and $B$ is finite free over $A[t, t^{-1}]$.
We need to show that $B_n = B/(1-t^nf)B$ is flat over $A$, for $n$ sufficiently large.
We may assume that $B$ has constant rank over $A[t,t^{-1}]$ and choose a basis $e_1, \dots, e_d$.
We can write \[ f e_i = \sum_j a_{ij} e_j, \] with $a_{ij} \in A[t,t^{-1}]$.
Thus for $n$ sufficiently large we have $t^n a_{ij} \in tA[t]$, for all $i,j$.
This implies that $(1-t^nf): B \to B$ is universally injective over $A$: if $A'$ is any $A$-module, then $(1-t^nf)$ is injective on \[ B' := B \otimes_A A' \wequi A'[t,t^{-1}]\{e_1, \dots, e_d\}. \]
Indeed we can write $0 \ne b' \in B'$ as \[ b' = \sum_{k \ge k_0} t^k b_k, \] with $b_k \in A'\{e_1, \dots, e_d\}$ and $b_{k_0} \ne 0$, and then \[ (1-t^nf)b' \in t^{k_0} b_{k_0} + t^{k_0+1}A'[t]\{e_1, \dots e_d\} \] is non-zero as well.\NB{I.e. $1-t^nf$ is unitriangular.}
The desired flatness now follows from \cite[Tags 058I and 058P]{stacks-project}.\NB{Or \cite[Proposition 11.3.7]{EGAIV}, or just look at exact sequence of $Tor$.}
\end{proof}
\begin{remark} \label{rmk:flatness-ext}
The argument shows that given $f_1, f_2 \in \scr O(Z)$ and $N$ sufficiently large, then \[ Z(1-t^n(f_1 t^a + f_2 t^b)) \to X \] is flat for all $n > N$ and $a, b \ge 0$.
\end{remark}

Define maps $g_n^+, g_n^-: \Gm \times \Gm \to \A^1$ via \[ g_n^+(t_1, t_2) = t_1^n + 1 \text{ and } g_n^-(t_1,t_2) = t_1^n + t_2. \]
Further define maps $\A^1 \times \Gm \times \Gm \to \A^1$ via \[ h_{mn}^\pm(s, t_1, t_2) = s g_n^\pm(t_1, t_2) + (1-s) g_m^\pm(t_1, t_2). \]
Given a span \[ X \times \Gm \leftarrow Z \to Y \times \Gm, \] put \[ Z_{mn}^{\pm} = Z(h_{mn}^\pm) \subset Z \times \A^1; \] note that there are induced spans \[ X \times \A^1 \leftarrow Z_{mn}^\pm \to Y. \]
\begin{corollary} \label{cor:exhaustive}
Let $X$ be smooth over a perfect field and suppose given a span \[ X \times \Gm \xleftarrow{p} Z \to Y \times \Gm \] with $p$ finite locally free.
Then for $N$ sufficiently large and $m,n>N$ the induced maps $Z_{mn}^\pm \to X \times \A^1$ are finite locally free.
\end{corollary}
\begin{proof}
Voevodsky's original argument \cite[Lemma 4.1]{voevodsky2002cancellation}, explained in slightly more detail in \cite[Lemma 4.20]{bachmann-notes}, shows that $Z_{mn}^\pm \to X \times \A^1$ is finite for $m,n$ sufficiently large.
It thus remains to establish flatness.
Note that for $m=n+r \ge m$ we have \[ h_{mn}^+ = t_1^n(s + (1-s)t_1^r) + 1 \] and \[ h_{mn}^- = t_1^n(s + (1-s)t_1^r) + t_2. \]
The result thus follows via Remark \ref{rmk:flatness-ext} from Lemma \ref{lemm:flatness} applied to $Z \times \A^1 \xrightarrow{p \times \A^1} X \times \Gm \times \A^1$, with \[ f = -(s + (1-s)t_1^r) \quad\text{and}\quad f = -(s + (1-s)t_1^r)t_2^{-1}, \] respectively.
The case $m<n$ is treated similarly.
\end{proof}

Given any span $\alpha: X \times \Gm \leftarrow Z \to Y \times \Gm$ (with $Z \to X$ not necessarily finite locally free), denote by $\rho_{mn}^\pm(\alpha)$ the span $X \times \A^1 \leftarrow Z_{mn}^\pm \to Y$.
One checks immediately (see \cite[Lemma 4.17]{bachmann-notes} for a more conceptual explanation) that if $\beta: X' \to X$ and $\gamma: Y \to Y'$ are arbitrary spans, then \begin{equation} \label{eq:compat} \rho_{mn}^\pm(\gamma \circ \alpha) \wequi \gamma_* \rho_{mn}^\pm(\alpha) \quad\text{and}\quad \rho_{mn}^\pm(\alpha \circ \beta) \wequi \beta^*\rho_{mn}^\pm(\alpha). \end{equation}
Let $G = \mu_+(\Gm)$.
For $Y \in \Sm_k$, define \[ F_i \Omega_G \Sigma_G \mu_+(Y) \subset \Omega_G \Sigma_G \mu_+(Y) \] as the subpresheaf of those spans $X \times \Gm \leftarrow Z \to Y \times \Gm$ such that for all $m,n \ge i$ the maps $Z_{mn}^\pm \to X$ are finite locally free.
It follows from \eqref{eq:compat} that \[ \gamma_* (F_i \Omega_G \Sigma_G \mu_+ Y(X)) \subset F_i \Omega_G \Sigma_G \mu_+ Y'(X) \quad\text{and}\quad \beta^* F_i \Omega_G \Sigma_G \mu_+ Y(X) \subset \beta^* F_i \Omega_G \Sigma_G \mu_+ Y(X'). \]
In particular $F_i \Omega_G \Sigma_G \mu_+ Y$ defines a presheaf on $\Cor^\flf(k)$.
By construction we have \[ F_i \Omega_G \Sigma_G \mu_+ Y \subset F_{i+1} \Omega_G \Sigma_G \mu_+ Y \subset \dots, \] and by Corollary \ref{cor:exhaustive} we have \[ \colim_i F_i \Omega_G \Sigma_G \mu_+ Y \wequi \Omega_G \Sigma_G \mu_+ Y. \]
For $m,n \ge i$ define maps \[ \tilde\rho_{mn,i}^\pm: F_i \Omega_G \Sigma_G \mu_+ Y \to \Omega_{\A^1_+} \mu_+ Y \] via $Z \mapsto Z_{mn}$.
By construction $Z_{mn}$ is finite locally free over $X \times \A^1$, and by \eqref{eq:compat} $\tilde\rho_{mn,i}^\pm$ defines a morphism in $\PSh_\Sigma(\Cor^\flf(k))$.
By adjunction we obtain morphisms \[ \rho_{mn,i}^\pm: \mu_+(\A^1) \wedge F_i \Omega_G \Sigma_G \mu_+ Y \to \mu_+ Y; \] i.e. we obtain $\A^1$-homotopies between the various \[ \rho_m^\pm: F_i \Omega_G \Sigma_G \mu_+ Y \to \mu_+ Y \] for $m \ge i$ (obtained by restriction to $0$ or $1$).
By construction, for $m \ge j \ge i$, the following diagram commutes
\begin{equation*}
\begin{tikzcd}
F_i \Omega_G \Sigma_G \mu_+ Y \ar[r] \ar[d, "\rho_m"] & F_j \Omega_G \Sigma_G \mu_+ Y \ar[dl, "\rho_m"] \\
\mu_+ Y.
\end{tikzcd}
\end{equation*}
It follows that in the following diagram, all cells commute up to $\A^1$-homotopy
\begin{equation*}
\begin{CD}
F_1 \Omega_G \Sigma_G \mu_+ Y @>>> F_2 \Omega_G \Sigma_G \mu_+ Y @>>> \dots @>>> \Omega_G \Sigma_G \mu_+ Y \\
@V{\rho_1^\pm}VV                   @V{\rho_2^\pm}VV            \\
\mu_+ Y                       @=   \mu_+ Y @= \cdots.
\end{CD}
\end{equation*}
Consequently after motivic localization we obtain induced maps in the colimit \[ \rho^\pm: L_\mot \Omega_G \Sigma_G \mu_+ Y \wequi \colim_i L_\mot F_i \Omega_G \Sigma_G \mu_+ Y \to L_\mot \mu_+ Y. \]
Group-completing and taking the difference yields \[ \rho := \rho^+ - \rho^-: L_\mot^\gp \Omega_G \Sigma_G \mu_+ Y \to L_\mot^\gp \mu_+ Y. \]
\begin{theorem} \label{thm:cancellation-flf}
The category $\Cor^\flf(k)$ satisfies cancellation.
\end{theorem}
\begin{proof}
We shall apply Proposition \ref{prop:abstract-cancel} with $\scr C$ the homotopy category of $\Spc^\flf(k)^\gp$; thus $\bar G = \mu(\Gm)$.
This object is symmetric since $\Gm$ is symmetric in $\Spc^\fr(k)^\gp$, being semi-invertible there (see also \cite[Lemma 3.3.3]{EHKSY}).
Thus assumption (4) holds.
Since $L_\mot^\gp \Omega_G \Sigma_G \mu_+ Y \wequi \Omega_G \Sigma_G L_\mot^\gp \mu_+ Y$ by Proposition \ref{prop:Omega-Gm-mot-loc}, our map $\rho$ takes the required form.
Assumptions (1) and (2) already hold for all the $\rho_m^\pm$, i.e. before any localization; assumption (2) is a special case of \eqref{eq:compat} and assumption (1) is equally formal.\todo{details?}
It remains to verify (3).
The composite \[ \mu_+ Y \xrightarrow{u_Y} F_i \Omega_\Gm \Sigma_\Gm \mu_+ Y \xrightarrow{\rho_i} \mu_+ Y \] is easily checked to send a span $\alpha$ to $\alpha \otimes (\rho_i u_\1 \id)$.
Consequently $\rho \bar u = \id \otimes \rho(\bar u_\1(\id))$.
Since $\bar u_\1(\id) = \id_G - p$, where $p: G \to * \to G$ is the projector onto the trivial summand, the result follows from Lemma \ref{lemm:cancel-final} below.
\end{proof}

\begin{lemma} \label{lemm:cancel-final}
For each $n > 0$ we have
\begin{enumerate}
\item $\rho_n^+(p) = \rho_n^-(p)$, and
\item $\rho_n^+(\id_G) \stackrel{\A^1}{\wequi} \rho_n^-(\id_G) + \id_\1$.
\end{enumerate}
\end{lemma}
\begin{proof}
This is essentially \cite[Lemma 4.3]{voevodsky2002cancellation}.

Note that $p$ is represented by the correspondence $G \xleftarrow{\wequi} G \xrightarrow{1} G$, so that by Construction, $\rho_n^\pm(p)$ is represented by $* \leftarrow Z(g_n^\pm(t,1)) \to *$
But $g_n^+(t,1) = g_n^-(t,1)$, whence (1).

Similarly $\rho_n^\pm(\id_G)$ is represented by $Z_\pm := Z(g_n^\pm(t,t))$, so $Z_+ = Z(t^n+1)$ and $Z_- = Z(t^n+t)$, where both $t^n+1, t^n+t$ are viewed as functions on $\A^1 \setminus 0$.
For a function $f: X \times \A^1 \to \A^1$ denote by $D(f)$ the span \[ X \leftarrow Z(f) \to *. \]
Consider the span $H = D(t^n + ts + 1-s)$, which is immediately verified to be finite flat.
Then $H$ provides an $\A^1$-homotopy between $D(t^n+1)$ and $D(t^n+t)$, where this time we view $t^n+1, t^n+t$ as functions on $\A^1$.
Now \[ Z(t^n+1|\A^1) = Z(t^n+1|\A^1 \setminus 0) = Z^+, \] whereas \[ Z(t^n+t|\A^1) = Z(t^n+t|\A^1 \setminus 0) \coprod \{0\} = Z^- \coprod \{0\}. \]
Since $* \leftarrow \{0\} \to *$ is the identity correspondence, $H$ provides the desired homotopy.

This concludes the proof.
\end{proof}

\subsection{Rational contractibility}
\begin{proposition}\label{prop:ratcontractible}
Let $X$ be a smooth connected scheme over $k$ and $x_0\in X$ be a rational point of $X$. Assume that there exists an open subscheme $W\subset X\times \A^1$ containing $(X\times \{0,1\})\cup (x_0\times \A^1)$ and a morphism of schemes $f:W\to X$ such that $f|_{X\times 0}=x_0$, $f|_{X\times 1}=\id_X$ and $f|_{x_0\times \A^1}=x_0$.
Then $\hflf{X,x_0}$ is rationally contractible (as a presheaf of monoids).
\end{proposition}
\begin{proof}
We follow closely Suslin's proof in \cite[Proposition 2.2]{suslin2003grayson}.
We shall construct a commutative diagram in $\CMon(\PSh(\Sm_k))$
\begin{equation*}
\begin{CD}
\hflfp{x_0} @>s_0>> \ratC \hflfp{x_0} \\
@VjVV              @VVV             \\
\hflfp{X}   @>s>>   \ratC \hflfp{X}
\end{CD}
\end{equation*}
such that $i_0^* s_0 \wequi \id_{x_0} \wequi i_1^* s_0$, $i_0^*s \wequi \id_X$, and $i_1^*s$ factors through $j$.
Since $\ratC$ commutes with colimits of monoids (it commutes with finite products, i.e. coproducts of monoids, and sifted colimits) the above diagram induces \[ \bar s: \hflf{X,x_0} \to \ratC \hflfp{X} \sslash \ratC \hflfp{x_0} \wequi \ratC(\hflf{X,x_0}), \] where $\sslash$ means quotient as presheaf of commutative monoids.
By construction, this exhibits $\hflf{X,x_0}$ as rationally contractible.

The morphism $s$ is constructed as follows.
Write $V \subset X \times \A^1$ for the closed complement of $W$.
Given $Y \in \Sm_k$ and a correspondence \[ \alpha = (Y \xleftarrow{p} Z \xrightarrow{q} X) \in \hflfp{X}(Y), \] denote by $p, q$ also the maps $Z \times \A^1 \to Y \times \A^1$ and $Z \times \A^1 \to X \times \A^1$.
Let $V' = q^{-1}(V) \subset Z \times \A^1$, $V'' = p(V')$, $U = Y \times \A^1 \setminus V''$ and $W' = p^{-1}(U)$.
Note that $V'$ does not contain any point above $0$ or $1$ and is closed.
Since $p$ is finite, $V''$ is also closed and contains no point above $0$ or $1$.
Thus $U$ is open and contains $Y \times \{0,1\}$.
Moreover $W' \subset Z \times \A^1 \setminus V' = q^{-1}(W)$.
There is thus a well-defined correspondence \[ s(\alpha) = (U \leftarrow W' \to W \xrightarrow{f} X) \in \ratC \hflfp{X}(Y). \]
This assignment is readily promoted to a morphism of presheaves.
The morphism $s_0$ just sends $Y \leftarrow Z \to x_0$ to $Y \times \A^1 \leftarrow Z \times \A^1 \to Z \to x_0$.
The required commutativity and factorization are readily established.\NB{give more details using Grothendieck construction etc?}
\end{proof}

\begin{proposition} \label{prop:flf-ratl-contr}
$\Cor^\flf(k)$ satisfies rational contractibility.
\end{proposition}
\begin{proof}
By Remark \ref{rmk:ratl-contr-simplified}, it suffices to show that $\hflf{(\A^1 \setminus 0)^{\times n},x_0}$ is rationally contractible as a presheaf of monoids.
As usual this follows from Proposition \ref{prop:ratcontractible} by taking $W$ to be defined by $ut_i + (1-u) \ne 0$ and $f(t_1, \dots, t_n, u) = u(t_1, \dots, t_n) + (1-u)x_0$.
\end{proof}

\section{Applications} \label{sec:applications}
\begin{theorem} \label{thm:main}
Let $k$ be a perfect field.
The functor \[ \Spc^\flf(k)^\gp \to \SH^\flf(k) \] is fully faithful.
\end{theorem}
\begin{proof}
Immediate from Theorem \ref{thm:cancellation-flf} and Proposition \ref{prop:cancellation}.
\end{proof}

Write \[ \mu: \SH(k) \adj \SH^\flf(k): \mu^* \] for the canonical adjunction.
Recall from \cite[\S5]{hoyois2020hilbert} the motivic spectrum $\kgl$ of (very) effective algebraic $K$-theory.
\begin{theorem} \label{thm:present-kgl}
Let $k$ be a perfect\NB{not necessary} field.
We have $\mu^*(\1_{\SH^\flf(k)}) \wequi \kgl$, and this spectrum is presented as a $\P^1$-$\Omega$-spectrum via \[ \kgl \wequi (L_\mot \hflfp{*}^\gp, L_\mot \hflf{\P^1}, L_\mot \hflf{(\P^1)^{\wedge 2}}, \dots). \]
\end{theorem}
\begin{proof}
It is immediate from Theorem \ref{thm:cancellation-flf} and Corollary \ref{cor:representing-spectrum} that $\mu^*(\1)$ has a presentation as claimed.
Moreover by Proposition \ref{prop:flf-ratl-contr} and Corollary \ref{cor:identify-unit} we have $\mu^*(\1) \wequi \Sigma^\infty_\fr \hflfp{*}$.
But by construction $\hflfp{*} = \scr F\!\scr F\mathrm{lat}_k$, so the remaining claim follows from \cite[Theorem 5.4]{hoyois2020hilbert}.
\end{proof}

\begin{corollary}
Let $k$ be perfect of exponential characteristic $e$.
We have \[ \SH^\flf(k)[1/e] \wequi \kgl[1/e]\Mod. \]
\end{corollary}
\begin{proof}
This is a formal consequence of compact-rigid generation (which is why we invert $e$ \cite[Corollary B.2]{levine2013algebraic}) and Theorem \ref{thm:present-kgl}; see e.g. the proof of \cite[Lemma 5.3]{bachmann-criterion}.
\end{proof}

\appendix
\section{Rational contractibility} \label{app:ratl-contr}
We review Suslin's notion of rational contractibility \cite[\S2]{suslin2003grayson} and extend the basic properties of this notion to presheaves of spaces.
All results in this section are essentially straightforward reformulations of Suslin's.
Throughout $k$ denotes a field (not necessarily perfect).

For a presheaf $F$ (of spaces) on $\Sm_k$, define a new presheaf $\ratC F$ on $\Sm_k$ by \[ (\ratC F)(X) = \colim_{X \times \{0,1\} \subset U \subset X \times \A^1} F(U); \] here the colimit is over open subschemes of $X \times \A^1$.
Note that pullback to $0$ or $1$ defines two maps of presheaves \[ i_0^*, i_1^*: \ratC F \to F. \]
\begin{definition} \label{def:ratl-contr}
A presheaf $F$ is called \emph{rationally contractible} if there is a map $s: F \to \ratC F$ such that $i_0^* s \wequi \id_F$ and $i_1^* s$ is constant (i.e. factors through the terminal presheaf $*$).
\end{definition}

\begin{example} \label{ex:contr-rat-contr}
$F$ is called $\A^1$-contractible if there is a map $H: \A^1 \times F \to F$ such that $i_0^* H \wequi \id_F$ and $i_1^* H$ is constant.
Equivalently there is a map $H': F \to \iMap(\A^1, F)$ with $i_0^* H' \wequi \id_F$ and $i_1^* H'$ constant.
Since there is a canonical map $\iMap(\A^1, F) \to \ratC F$, we see that $\A^1$-contractible presheaves are rationally contractible.
\end{example}

\begin{example} \label{ex:product-contr}
We have $\ratC(F \times G) \wequi \ratC(F) \times \ratC(G)$.
It follows that if $F, G$ are rationally contractible then so is $F \times G$.
\end{example}

\begin{example} \label{ex:retract-contr}
Consider a retraction $G \xrightarrow{\alpha} F \xrightarrow{\beta} G$ and let $s: F \to \ratC F$ exhibit $F$ as rationally contractible.
Then one easily checks that $G \xrightarrow{\alpha} F \xrightarrow{s} \ratC F \xrightarrow{\ratC \beta} \ratC G$ exhibits $G$ as rationally contractible.
In other words, rationally contractible presheaves are stable under retracts.
\end{example}

Write $\hat\Delta^\bullet = \hat\Delta^\bullet_k$ for the standard cosimplicial semilocal scheme \cite[\S5.1]{levine2008homotopy}.
Recall that the category $\EssSm_k$ embeds into the category of pro-objects in $\Sm_k$ \cite[Proposition 8.13.5]{EGAIV}.
Thus $F(\hat\Delta^\bullet)$ makes sense and is a simplicial space, and we write $|F(\hat\Delta^\bullet)|$ for its geometric realization.

\begin{lemma} \label{lemm:ratr-contr-real}
Let $F$ be rationally contractible.
Then $|F(\hat\Delta^\bullet)|$ is contractible.
\end{lemma}
\begin{proof}
Write $r: \Delta \to \EssSm_k$ for the cosimplicial object $\hat\Delta^\bullet$, and also for the left Kan extension $\PSh(\Delta) \to \PSh(\EssSm_k)$.
There are canonical maps \[ r(\Delta^1 \times \Delta^n) \to r(\Delta^1) \times r(\Delta^n) = \hat\Delta^1 \times \hat\Delta^n \to \hat\Delta^1 \hat\times \hat\Delta^n =: r'(\Delta^n), \] where $\hat\Delta^1 \hat\times \hat\Delta^n$ denotes the semilocalization of $\A^1 \times \A^n$ in the vertices.
This way we obtain a morphism of simplicial spaces\NB{first equivalence slightly non-trivial!} \[ s': r^*F \xrightarrow{r^*s} r^*\ratC F \wequi r'^*F \to r(\Delta^1 \times \Delta^\bullet)^*(F) \wequi \iMap(\Delta^1, r^*F). \]
By construction $i_0^*s' \wequi \id_{r^*F}$ and $i_1^*s'$ is constant; thus $r^*F \wequi F(\hat\Delta^\bullet)$ is simplicially contractible and hence $|r^*F| \wequi |F(\hat\Delta^\bullet)|$ is contractible.\footnote{Indeed by adjunction we obtain $H: \Delta^1 \times r^*F \to r^*F$, and since geometric realization commutes with finite products (being a sifted colimit \cite[Lemma 5.5.8.11, Remark 5.5.8.11]{lurie-htt}), $|H|: |\Delta^1 \times r^*F| \wequi |\Delta^1| \times |r^*F| \to |r^*F|$ is a contracting homotopy.}
\end{proof}

\begin{lemma} \label{lemm:basechange-ratl-contr}
Let $p: \Spec(K) \to \Spec(k)$ be a separable (not necessarily algebraic) field extension.
Then for $F \in \PSh(\Sm_k)$ we have $p^* \ratC F \wequi \ratC p^* F$.
In particular if $F$ is rationally contractible then so is $p^* F$.
\end{lemma}
\begin{proof}
The proof of \cite[Lemma 2.2]{bachmann-criterion} goes through unchanged.
\end{proof}

\begin{lemma} \label{lemm:Hom-contr}
Let $F$ be rationally contractible, and assume that $F$ promotes to a presheaf of grouplike commutative monoids.\NB{is this necessary?}
Then $\iMap(\A^n, F)$ is rationally contractible.
\end{lemma}
\begin{proof}
Since $* \to \A^n$ admits a retraction, the group structure allows us to split the fibration sequence \[ \iMap_*(\A^n, F) \to \iMap(\A^n, F) \to F, \] i.e. $\iMap(\A^n, F) \wequi \iMap_*(\A^n, F) \times F$.
By Examples \ref{ex:product-contr} and \ref{ex:contr-rat-contr} it thus suffices to show that $\iMap_*(\A^n, F)$ is $\A^1$-contractible, which is well-known (an $\A^1$-homotopy contracting $\A^n$ to the base point induces an $\A^1$-homotopy contracting $\iMap_*(\A^n, F)$ to its base point).
\end{proof}

\begin{corollary} \label{cor:ratl-contr-crit}
Let $F$ be a presheaf of grouplike commutative monoids on $\Sm_k$ such that the underlying presheaf of spaces is rationally contractible.
Let $K$ be the field of fractions of a smooth connected $k$-scheme.
Then \[ |(L_{\A^1} F)(\hat\Delta^\bullet_K)| \wequi *. \]
\end{corollary}
\begin{proof}
Write $p: \Spec(K) \to \Spec(k)$ for the base change.
Since $p^*$ commutes with $L_{\A^1}$ (e.g. by \cite[Lemma A.4]{hoyois-algebraic-cobordism}) and $\ratC$ (by Lemma \ref{lemm:basechange-ratl-contr}), we may assume that $K=k$.
Since colimits commute we find that \[ |(L_{\A^1} F)(\hat\Delta^\bullet_K)| \wequi \colim_{n \in \Delta^\op} (L_{\A^1} F)(\hat \Delta^n) \wequi \colim_{n \in \Delta^\op} \colim_{m \in \Delta^\op} F(\A^m \times \hat\Delta^n) \wequi \colim_{m \in \Delta^\op} \colim_{n \in \Delta^\op} \iMap(\A^m, F)(\hat \Delta^n). \]
By Lemmas \ref{lemm:Hom-contr} and \ref{lemm:ratr-contr-real}, each of the inner colimits is contractible, and hence so is the total colimit, $\Delta$ being sifted and hence contractible \cite[Proposition 5.5.8.7]{lurie-htt}.
\end{proof}

If $F$ is a presheaf of commutative monoids then so is $\ratC F$.
Indeed the colimit defining $(\ratC F)(X)$ is filtered, and hence may be computed in commutative monoids or spaces, with the same result.
We say that $F$ is rationally contractible as a presheaf of commutative monoids if there is a morphism of presheaves of commutative monoids $s: F \to \ratC F$ such that $i_0^*s \wequi \id_F$ and $i_1^*s \wequi *$, as morphisms of presheaves of commutative monoids.
Clearly if $F$ is rationally contractible as a presheaf of commutative monoids then it is also rationally contractible in the previous sense.
\begin{lemma} \label{lemm:ratl-contr-gp}
On presheaves of commutative monoids, $\ratC$ commutes with group completion.
In particular if $F$ is rationally contractible as a presheaf of commutative monoids, then $F^\gp$ is also rationally contractible (as a presheaf of commutative monoids).
\end{lemma}
\begin{proof}
It suffices to show that group completion of commutative monoids preserves filtered colimits and final objects.
The first statement holds since group completion is a localization and grouplike monoids are closed under filtered\NB{in fact all} colimits, and the second statement is obvious since the final commutative monoid is grouplike.
\end{proof}

\bibliographystyle{alpha}
\bibliography{bibliography}

\Addresses

\end{document}